\documentclass[12 pt]{amsart}
\usepackage{amsthm}
\usepackage{latexsym}
\usepackage{amsmath}
\usepackage{amssymb}
\usepackage{young}
\usepackage{youngtab}

\usepackage{fullpage}
\usepackage{amscd,amsthm}
\usepackage{comment}
\usepackage{color}

\usepackage{tikz} \usetikzlibrary {positioning}

\theoremstyle{plain}
\newtheorem{theorem}{Theorem}
\newtheorem{lemma}[theorem]{Lemma}

\newtheorem{proposition}[theorem]{Proposition}

\theoremstyle{definition}
\newtheorem{definition}[theorem]{Definition}
\newtheorem{example}[theorem]{Example}

\newtheorem{question}[theorem]{Question}

\theoremstyle{remark}
\newtheorem{remark}[theorem]{Remark}

\begin{document}

\title{Face rings of cycles, associahedra, \\
and standard Young tableaux}

\author{Anton Dochtermann}
\date\today

\maketitle
\begin{abstract}
We show that $J_n$, the Stanley-Reisner ideal of the $n$-cycle, has a free resolution supported on the $(n-3)$-dimensional simplicial associahedron $A_n$.  This resolution is not minimal for $n \geq 6$; in this case the Betti numbers of $J_n$ are strictly smaller than the $f$-vector of $A_n$.  We show that in fact the Betti numbers $\beta_{d}$ of $J_n$ are in bijection with the number of standard Young tableaux of shape $(d+1, 2, 1^{n-d-3})$.    This complements the fact that the number of $(d-1)$-dimensional faces of $A_n$ are given by the number of standard Young tableaux of (super)shape $(d+1, d+1, 1^{n-d-3})$; a bijective proof of this result was first provided by Stanley.   An application of discrete Morse theory yields a cellular resolution of $J_n$ that we show is minimal at the first syzygy.   We furthermore exhibit a simple involution on the set of associahedron tableaux with fixed points given by the Betti tableaux, suggesting a Morse matching and in particular a poset structure on these objects.

\end{abstract}

\section{Introduction}

In this paper we study some intriguing connections between basic objects from commutative algebra and combinatorics.  For ${\mathbb K}$ an arbitrary field we let $R = {\mathbb K}[x_1, x_2, \dots, x_n]$ denote the polynomial ring in $n$ variables.  We let $J_n$ denote the edge ideal of the complement of the $n$-cycle $C_n$.   By definition, $J_n$ is the ideal generated by the degree 2 monomials corresponding to the diagonals of $C_n$.  One can also realize $J_n$ as the Stanley-Reisner ideal of the cycle $C_n$ (now thought of as a one-dimensional simplicial complex).

\begin{figure}[!ht]
\begin{center}
  \includegraphics[scale = 0.55]{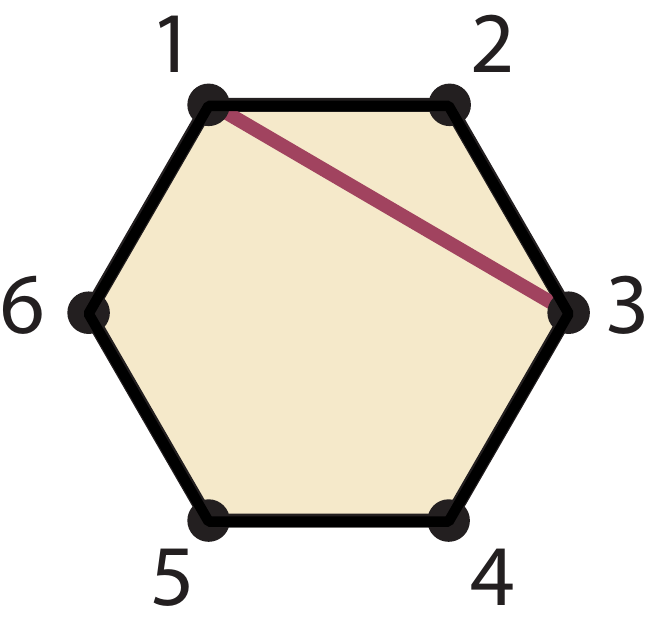}
\end{center}
    \caption{$J_6 = \langle  x_1x_3, x_1x_4, x_1x_5, x_2x_4, x_2x_5, x_2x_6, x_3x_5, x_3x_6, x_4x_6 \rangle$.}
\label{Fig:cycle}
\end{figure}

The ideals $J_n$ are of course very simple algebraic objects and their homological properties are well-understood.  One can verify that $R/J_n$ is a Gorenstein ring, the dimension of $R/J_n$ is 2, and (hence) the projective dimension of $R/J_n$ is $n-2$.  In fact a minimal free resolution can be described explicitly, and cellular realizations have been provided by Biermann \cite{Bie} and more recently by Sturgeon \cite{Stu}.

We wish to further investigate the combinatorics involved in the resolutions of $J_n$.  Our original interest in cellular resolutions of $J_n$ came from the fact that the ideal $J_n$ has an `almost linear' resolution, in the sense that the nonzero entries in the differentials of its minimal resolution are linear forms from $R$, \emph{except} at the last syzygy where the nonzero entries are all degree 2.  Recent work in combinatorial commutative algebra has seen considerable interest in cellular resolutions of monomial and binomial ideals (see for example \cite{BayStu}, \cite{BraBroKle}, \cite{DocEng}, \cite{DocMoh}, \cite{Goo}, \cite{Mer}, \cite{NagRei}, \cite{Sin}), but in almost all cases the ideals under consideration have linear resolutions.  Here we seek to extend some of these constructions.

In the construction of any cellular resolution, one must construct a $CW$-complex with faces labeled by monomials that generate the ideal.  In the case of $J_n$ there is a well known geometric object whose vertices are labeled by the diagonals of an $n$-gon, namely the (simplicial) associahedron $A_n$.  By definition, $A_n$ is the simplicial complex with vertex set given by diagonals of an $n$-gon, with faces given by collections of diagonals that are non-crossing.  The facets of $A_n$ are triangulations of the $n$-gon, of which there are a Catalan number many.  It is well known that $A_n$ is spherical, and in fact can be realized as the boundary of a convex polytope.  In addition there is a natural way to associate a monomial to each face of $A_n$, and in the first part of the paper we show that this labeled facial structure of $A_n$ (considered as a polytope with a single interior cell) encodes the syzygies of $J_n$. 
\newtheorem*{thm:resolution}{Theorem \ref{thm:resolution}}
\begin{thm:resolution}
With its natural monomial labeling, the complex $A_n$ supports a free resolution of the ideal $J_n$.
\end{thm:resolution}
The resolution of $J_n$ supported on the associahedron $A_n$ is not minimal for $n > 5$, and in particular in this case we have faces $F \subsetneq G$ with the same monomial labeling.  The $f$-vector of $A_n$ is completely understood (a closed form can be written down), and in fact the number of $(d-1)$-dimensional faces of $A_n$ is equal to the number of standard Young tableaux of shape $(d+1, d+1, 1^{n-d-3})$; a bijective proof of this was first provided by Stanley \cite{Sta}.  

Since a resolution of $J_n$ is supported on $A_n$ we know that the $f$-vector of $A_n$ provides an upper bound on the Betti numbers $\beta_d(R/J_n)$, with equality in the case of $\beta_1(R/J_n) = f_0(A_n)$.  In the second part of the paper we show that the Betti numbers $\beta_d$ of $R/J_n$ are given by standard Young tableaux on a set of subpartitions involved in the Stanley bijection.  
\newtheorem*{thm:Betti}{Theorem \ref{thm:Betti}}
\begin{thm:Betti}
The total Betti numbers $\beta_d$ of the module $R/J_n$ are given by the number of standard Young tableaux of shape $(d+1, 2, 1^{n-d-3})$.
\end{thm:Betti}
\noindent
This bijection along with an application of the hook formula leads to a closed form expression for the Betti numbers of $R/J_n$.  In addition, the fact that the partition $(d+1, 2, 1^{n-d-3})$ is conjugate to $(n-d-1, 2,1^{d+1})$ provides a nice combinatorial interpretation of the palindromic property $\beta_d = \beta_{n-d-2}$ for the Betti numbers of the Gorenstein ring $R/J_n$.

The fact that we can (in theory) identify the Betti numbers of $R/J_n$ with certain faces of $A_n$ suggests that it may be possible to collapse away faces of $A_n$ to obtain a minimal resolution of $J_n$, employing an algebraic version of Morse theory due to Batzies and Welker (\cite{BatWel}).  Indeed certain geometric properties of any subdivision of an $n$-gon (along with the almost linearity of $J_n$) imply that certain faces must be matched away.   For $d=2$ we are able to write down a Morse matching involving the edges and 2-faces of $A_n$ such that the number of unmatched (critical) cells is precisely $\beta_2$ (corresponding to the first syzygy module of $R/J_n$), see Proposition \ref{prop:edges}.    This leads to minimal resolutions of $J_n$ for the cases $n \leq 7$.

In addition, our identification of both the Betti numbers of $R/J_n$ and the faces of $A_n$ with standard Young tableaux leads us to consider a partial matching on the set of associahedron tableaux such that the unmatched elements correspond to the Betti numbers.  The hope would be to import a poset structure from the face poset of $A_n$ to extend this matching to a Morse matching.   The trouble with this last step is that the Stanley bijection does not give us an explicit labeling of the faces of $A_n$ by standard Young tableaux; there are choices involved and the bijection itself is recursively defined.  However, we can define a very simple partial matching on the set of standard Young tableaux of shape $(d+1, d+1, 1^{n-d-3})$ such that the unmatched elements can naturally be thought of as standard Young tableaux of shape $(d+1, 2, 1^{n-d-3})$ (by deleting the largest entries); see Proposition \ref{prop:matching}.  This suggests a poset structure on the set of standard Young tableax that extends this covering relation.    

The rest of the paper is organized as follows.  We begin in Section \ref{sec:algebra} with some basics regarding the commutative algebra involved in our study.  In Section \ref{sec:Ass} we discuss associahedra and their role in resolutions of $J_n$.  We turn to standard Young tableaux in Section \ref{sec:tableaux} and here establish our results regarding the Betti numbers of $R/J_n$.   In Section \ref{sec:matching} we discuss our applications of discrete Morse theory and related matchings of stand Young tableaux.  We end with some open questions.

\section{Some commutative algebra}\label{sec:algebra}

As above we let $J_n$ denote the Stanley-Reisner ideal of the $n$-cycle, by definition the ideal in $R = {\mathbb K}[x_1, x_2, \dots, x_n]$ generated by degree 2 monomials corresponding to the diagonals.   We are interested in combinatorial interpretations of certain homological invariants of $J_n$, and in particular the combinatorial structure of its minimal free resolution.  Recall that a {\em free resolution} of an $R$-module $M$ is an exact sequence of $R$-modules
\[0 \leftarrow M \leftarrow F_1 \leftarrow F_2 \leftarrow \dots \leftarrow F_p \leftarrow 0,\]
\noindent
where each $F_d \simeq \oplus_j R(-j)^{\beta_{d,j}}$ is free and the differential maps are graded.  The resolution is {\em minimal} if each of the $\beta_{d,j}$ are minimum among all resolutions, in which case the $\beta_{d,j}$ are called the (graded) Betti numbers of $M$.  Also in this case the number $p$ (length of the minimal resolution) is called the \emph{projective dimension} of $M$.

Our main tool in calculating Betti numbers will be Hochster's formula (see for example \cite{MilStu}), which gives a formula for the Betti numbers of the Stanley-Reisner ring $R/I_\Delta$ associated to a simplicial complex $\Delta$.

\begin{theorem}[Hochster's formula] \label{Hochster}
For a simplicial complex $\Delta$ on vertex set $[n]$ we let $R/I_\Delta$ denote its Stanley-Reisner ring.  Then for $d \geq 1$ the Betti numbers $\beta_{d,j}$ of $R/I_\Delta$ are given by
\begin{equation} \label{eq:Hochster}
\beta_{d,j}(\Delta)=\sum_{W\in {[n] \choose j}} \dim_k \tilde{H}_{j-d-1} (\Delta[W];k).
\end{equation}
\noindent
Here $\Delta[W]$ denotes the simplicial complex induced on the vertex set $W$.
\end{theorem}

A \emph{cellular resolution} of $M$ is a CW-complex ${\mathcal X}$ with a monomial labeling of its faces, such that the algebraic chain complex computing the cellular homology of ${\mathcal X}$ `supports' a resolution of $M$. We refer to Section \ref{sec:Ass} for details and more precise definitions.  

We next collect some easy observations regarding the Betti numbers of $J_n$.  Since $J_n$ is the Stanley-Reisner ideal of a triangulated 1-dimensional sphere, we see that $R/J_n$ is Gorenstein and has (Krull) dimension 2.  The Auslander-Buchsbaum formula then implies that the projective dimension of $R/J_n$ is $n-2$, which says that $\beta_{d,j} = 0$ whenever $d > n-2$.  An easy application of Hochster's formula also implies that a minimal resolution of $J_n$ is linear until the last nonzero term, by which we mean if $0 \leq d < n-2$,  then $\beta_{d,j} = 0$ for all $j \neq d+1$.  Also, we have $\beta_{n-2,n} = 1$ and $\beta_{n-2,j} = 0$ for $j \neq n$.  In this sense the ideals $J_n$ have an `almost linear' resolution, as mentioned in the introduction.

{\bf Convention}: Since for any $d$ we have $\beta_{d,j} \neq 0$ for at most one value of $j$, we we will (without loss of generality) sometimes drop the $j$ and use $\beta_d = \beta_{d,j}$ to denote the Betti numbers of $J_n$.

\section{Asssociahedra}\label{sec:Ass}

For each $n$, we let $A_n$ denote the (dual) \emph{associahedron}, the $(n-4)$-dimensional simplicial complex whose vertices are given by diagonals of a labeled regular $n$-gon, with facets given by triangulations (collections of diagonals that do not intersect in their interior).  It is well known that $A_n$ is homeomorphic to a sphere and in fact is polytopal, and several embeddings (most often of the dual simple polytope) are described throughout the literature (see \cite{CebZie} for a good account of the history).  From here on we will use $A_n$ to denote the $(n-3)$-dimensional simplicial polytope (i.e., including the interior).

We wish to describe a monomial labeling of the faces of $A_n$.  Recall that each vertex of $A_n$ corresponds to some diagonal $\{i,j\}$ of an $n$-gon, so we simply label that vertex with the monomial $x_ix_j$.  We label the higher-dimensional faces of $A_n$ with the least common multiple of the vertices contained in that face.  We wish to show that, with this simple labeling, the associahedron $A_n$ supports a resolution of $J_n$.

\begin{figure}[!ht]
\begin{center}
  \includegraphics[scale = .36]{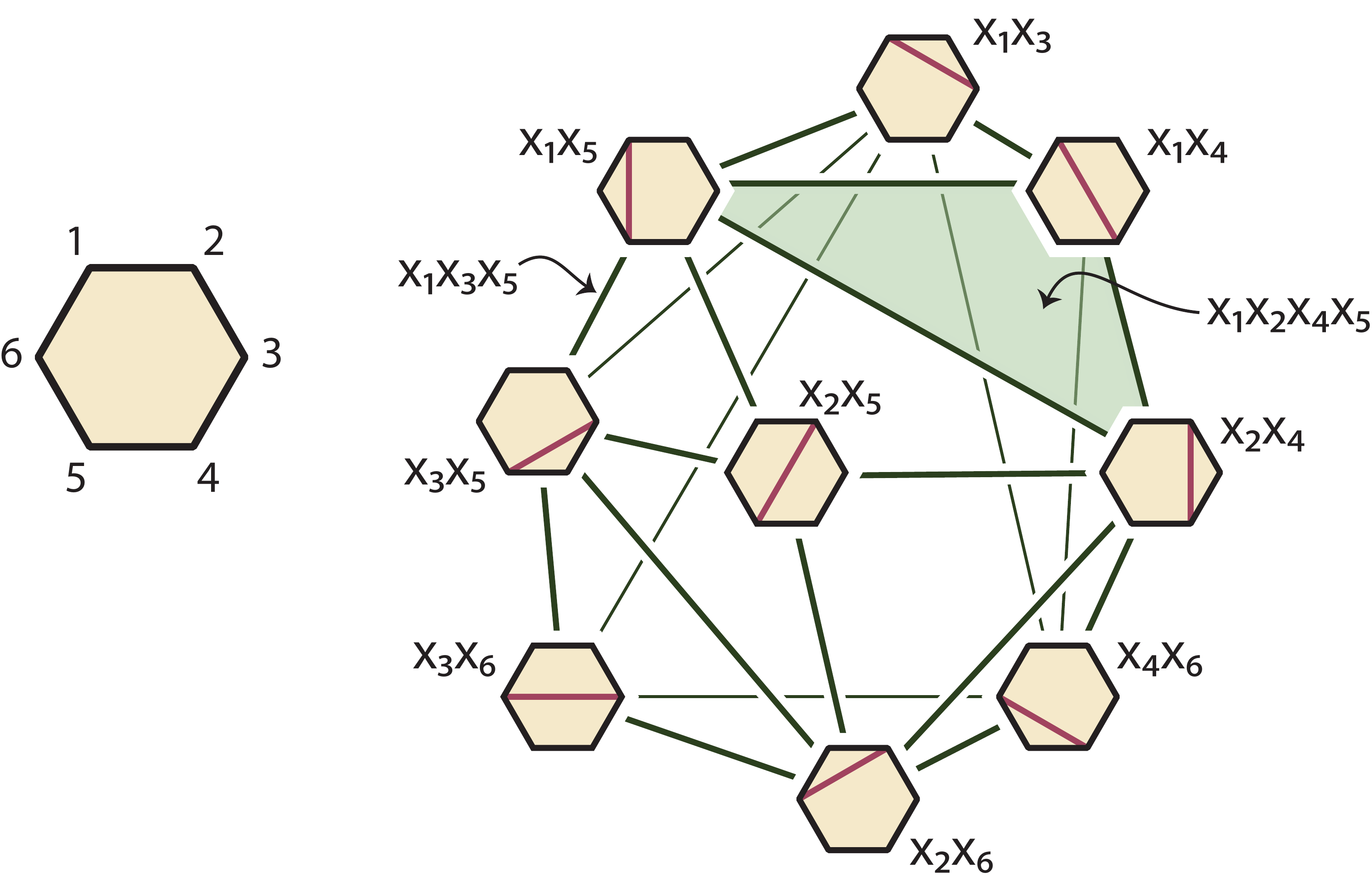}
\end{center}
    \caption{The complex $A_6$ with its monomial labeling (partially indicated).}
\label{Fig:resolution}
\end{figure}



Let us first clarify our terms.  To simplify notation we will associate to any monomial $x_1^{i_1}x_2^{i_2} \cdots x_n^{i_n} \in R$ the vector $(i_1, i_2, \dots, i_n) \in {\mathbb N}^n$ and will freely move between notations. We define a \emph{labeled polyhedral complex} to be a polyhedral complex $\mathcal{X}$ together with an assignment $a_F \in {\mathbb N}^n$ to each face $F \in {\mathcal X}$ such that for all $i = 1, 2, \dots, n$ we have
\[(a_F)_i = \max\{(a_G)_i: G \subset F \}.\]
If ${\mathcal X}$ is a labeled polyhedral complex we can consider the ideal $M = M_{\mathcal X} \subset {\mathbb K}[x_1, x_2, \dots, x_n]$ generated by the monomials corresponding to its vertices (as usual we identify an element $\alpha \in {\mathbb N}^n$ as the exponent vector of a monomial).  The topological space underlying ${\mathcal X}$ (with a chosen orientation) has an associated chain complex ${\mathcal F}_{\mathcal X}$ of $k$-vector spaces that computes cellular homology.  Since ${\mathcal X}$ has monomial labels on each of its cells, we can homogenize the differentials with respect to this basis and in this way ${\mathcal F}_{\mathcal X}$ becomes a complex of free modules over the polynomial ring $R = {\mathbb K}[x_1, x_2, \dots, x_n]$.   We say that the polyhedral complex ${\mathcal X}$ \emph{supports} a resolution of the ideal $M$ if ${\mathcal F}_{\mathcal X}$ is in fact a graded free resolution of $M$.     For more details and examples of cellular resolutions we refer to \cite{MilStu}.

For any $\sigma \in {\mathbb N}^n$ we let ${\mathcal X}_{\leq \sigma}$ denote the subcomplex of ${\mathcal X}$ consisting of faces $F$ for which $a_F \leq \sigma$ componentwise.  We then have the following criteria, also from \cite{MilStu}.  

\begin{lemma} \label{lem:criterion}
Let ${\mathcal X}$ be a labeled polyhedral complex and let $M = M_{\mathcal X} \subset {\mathbb K}[x_1, \dots ,x_n]$ denote the associated monomial ideal generated by the vertices. Then ${\mathcal X}$ supports a cellular resolution of $M$ if and only if the complex ${\mathcal X}_{\leq \sigma}$ is ${\mathbb K}$-acyclic (or empty) for all $\sigma \in {\mathbb N}^n$.  Futhermore, the resolution is minimal if and only if $a_F \neq a_G$ for any pair of faces $F \subsetneq G$.

\end{lemma}

With this criteria in place we can establish the following.

\begin{theorem}\label{thm:resolution}
For each $n \geq 4$ the associahedron $A_n$, with the monomial labeling described above, supports a cellular resolution of the edge ideal $J_n$.
\end{theorem}

\begin{proof}
Let $A_n$ denote the $n$-dimensional simplicial associahedron with this monomial labeling.  By construction, the vertices of $A_n$ correspond to the generators of $J_n$.   To show that $A_n$ supports a resolution of $J_n$, according to Lemma \ref{lem:criterion} it is enough to show that for any $\sigma \in {\mathbb N}^n$ we have that the subcomplex $(A_n)_{\leq \sigma}$ is ${\mathbb K}$-acyclic. 

Let $\sigma \in {\mathbb N}^n$ and let $(A_n)_{\leq \sigma}$ denote the subcomplex of $A_n$ consisting of all faces with a monomial labeling that divides $\sigma$ (as usual, thinking of $\sigma$ as the exponent vector of the monomial $x_1^{\sigma_1}x_2^{\sigma_2} \cdots x_n^{\sigma_n}$).  In particular, a face $F \in A_n$ is an element of $(A_n)_{\leq \sigma}$ if and only if for every diagonal $x_i x_j \in F$ we have $\sigma_i > 0$ and $\sigma_j >0$.  We claim that $(A_n)_{\leq \sigma}$ is contractible (and hence ${\mathbb K}$-acyclic).

Note that since $J_n$ is squarefree we may assume $\sigma$ has 0/1 entries, and hence we can identify $\sigma$ with a subset of $[n]$.  Also, if $\sigma_i = 1$ for all $i$ (so that $\sigma = [n]$) then we have $(A_n)_{\leq \sigma} = A_n$, which is a convex polytope and hence contractible.  If $\sigma$ has fewer than 2 nonzero entries then $(A_n)_{\leq \sigma}$ is empty.  Without loss of generality, we may then assume that $\sigma_1 = 1$ and $\sigma_n = 0$.  Let $j$ be the largest integer such that $j > 2$ and $\sigma_j = 1$.

Now, since $j < n$ and $\{1,j\} \subset \sigma$ we see that the diagonal $(1,j)$ is a vertex of the simplicial complex $(A_n)_{\leq \sigma}$.  In fact, $(1,j)$ is an element of every facet of $(A_n)_{\leq \sigma}$ since no other diagonal picked up by the elements of $\sigma$ intersects $(1,j)$.  We conclude that $(A_n)_{\leq \sigma}$ is a cone and hence contractible.

\end{proof}

For $n = 5$ one can check that this resolution is in fact \emph{minimal}, but for $n \geq 6$ this is no longer the case.  In particular for $n \geq 6$ we have faces $F \subsetneq G$ in $A_n$ with the same monomial label.



\section{Standard Young Tableaux}\label{sec:tableaux}

It turns out that the number of $i$-dimensional faces of the associahedron $A_n$ (the entries of the face vector of $A_n$) are given by the number of standard Young tableau (SYT) of certain shapes.  Recall that if $\lambda = (\lambda_1, \lambda_2, \dots, \lambda_k)$ is a partition of $n$, a \emph{standard Young tableaux} of shape $\lambda$ is a filling of the Young diagram of $\lambda$ with distinct entries $\{1,2,\dots, n\}$ such that rows and columns are increasing (see Example \ref{ex:tableaux}).

For $0 \leq d \leq n-3$, we let $f(n,d)$ denote the number of ways to choose $d$ diagonals in a convex $n$-gon such that no two diagonals intersect in their interior.  We see that $f(n,d)$ is precisely the number of $(d-1)$-dimensional faces of the polytope $A_n$.  A result attributed to Cayley (according to \cite{Sta}) asserts that 
\begin{equation}\label{eq:faces}
f(n,d) = \frac{1}{n+d}{{n+d}\choose{d+1}} {{n-3}\choose{d}}.
\end{equation} 

Using the hook length formula one can see that this number is also the number of standard Young tableaux of shape
\[(d+1, d+1, 1^{n-d-3}),\]
where as usual
\[1^{n-d-3} = \underbrace{1,\dots, 1}_{(n-d-3)\textrm{-times}}\]
denotes a sequence of $n-d-3$ entries with value 1.  This fact was apparently first observed by O'Hara and Zelevinsky (unpublished), and a simple bijection was given by Stanley \cite{Sta}.

\begin{example}
If we take $d = n-3$ we obtain
\[f(n,n-3) = \frac{1}{2n-3}{{2n-3}\choose{n-2}}{{n-3}\choose{n-3}} = \frac{1}{(n-2)+1}{{2(n-2)} \choose {n-2}},\]
the $(n-2)$nd Catalan number.
\end{example}

\begin{example}\label{ex:tableaux}
If $n = 5$ and $d=1$ the $f(5,1) = \frac{1}{6} {6 \choose 2} {2 \choose 1} = 5$ standard Young tableaux of shape $\lambda = (2,2,1)$ are given by
\medskip
\[\young(12,34,5) \hspace{.3 in} \young(12,35,4)  \hspace{.3 in} \young(13,24,5)  \hspace{.3 in} \young(13,25,4)  \hspace{.3 in} \young(14,25,3) \]
\medskip
These correspond to the 5 diagonals of a 5-gon.
\end{example}


It turns out the Betti numbers of the rings $R/J_n$ are also counted by the number of standard Young tableaux of certain related (sub)shapes.   To establish this result we will employ Hochster's formula (Theorem \ref{Hochster} from above).   Recall that the ring $R/J_n$ can be recovered as the Stanley-Reisner ring of the $n$-cycle, thought of as a 1-dimensional simplicial complex.  Note that when $|W| < n$ the only nonzero contribution to Equation (\ref{eq:Hochster}) comes from 0-dimensional reduced homology, i.e. the number of connected components of the induced complex on $W$ (minus one).

For $n \geq 4$, let $\beta^n_{d,j}$ denote the Betti numbers of the ring $R/J_n$.  Equation (\ref{eq:Hochster}) implies that for $d \geq 1$ we have $\beta^n_{d,j} = 0$ unless $d=n-2$ and $j = n$, or $1 \leq d < n-2$ and $j = d+1$.  Another application of Equation (\ref{eq:Hochster}) gives $\beta^n_{1,2} = {n \choose 2} - n$ and $\beta^n_{n-2,n} = 1$.  For the remaining cases we have the following result.

\begin{theorem} \label{thm:Betti}
For all $n \geq 5$ and $1 \leq d \leq n-3$ the Betti numbers of $R/J_n$ are given by
\begin{equation} \label{eq:betti}
\beta^n_d = \beta^n_{d,d+1} =  \textrm{the number of standard Young tableau of shape $(d+1 ,2, 1^{n-d-3})$}.
\end{equation}
\end{theorem}


\begin{proof}
We will establish the equality in Equation (\ref{eq:betti}) by showing that for $n \geq 6$ and $1 < d < n-4$ both sides of the equation satisfy the recursion 
\begin{equation} \label{eq:recursion}
F(n,d) = F(n-1,d-1) + F(n-1,d) + {{n-2} \choose {d+1}}.
\end{equation}
For the Betti numbers (the left hand side), we use Hochster's formula.  For each $d$, the computation of $\beta^n_{d,d+1}$ via Equation (\ref{eq:Hochster}) involves subcomplexes given by subsets $W$ of $[n]$ of size $d+1$.  First suppose we have chosen $W \subset [n]$ with $n \in W$.  Then we recover the contribution to Equation (\ref{eq:Hochster}) from the homology of induced subsets of size $d$ in the cycle on the vertices $[n-1]$, namely $\beta^{n-1}_{d-1, d}$.  However, if 1 and $n-1$ are both \emph{not} in $W$ then we get an additional contribution given by the isolated point $n$.  There are ${{n-3} \choose {d}}$ such instances.

Next suppose $n \notin W$.  Then again we recover contribution from the homology of induced subsets of size $d+1$ in the cycle $[n-1]$; this quantity is given by $\beta^{n-1}_{d,d+1}$.  In this case we have an additional contribution coming from the subsets $W$ including both 1 and $n-1$, since as subsets of the $n$-cycle these will be disconnected.  There are ${{n-3} \choose {d-1}}$ of these.  Putting this together, we have 
\[\beta^n_{d,d+1} = \beta^{n-1}_{d-1,d} + \beta^{n-1}_{d,d+1} + {{n-3} \choose d} + {{n-3} \choose {d-1}} = \beta^{n-1}_{d-1,d} + \beta^{n-1}_{d,d+1} + {{n-2} \choose {d}},\]
recovering Equation (\ref{eq:recursion}).

We next consider the right hand side of Equation (\ref{eq:betti}), namely the number of standard Young tableaux of shape $(d+1 ,2, 1^{n-d-3})$.  Recall that the fillings involve picking entries one each from the set $\{1,2,\dots,n\}$.  If $n$ is an entry in the first row (necessarily in the last column) then we recover all such fillings from standard Young tableaux of shape
\[(d, 2, 1^{n-d-3}) = ((d-1)+1, 2, 1^{(n-1)-(d-1)-3}).\]
 If $n$ is the (only) entry in the last row, then we recover all such fillings from standard Young tableaux of shape $(d+1,2,1^{n-d-4}) = (d+1,2,1^{(n-1)-d-3})$.  With these counts we miss the standard tableaux with $n$ as the entry in the second row (necessarily in the second column).  In this case we must have 1 as the entry in the first row, first column, but are free to choose any increasing sequence of length $d$ to fill the remaining entries of the first row (with the rest of the entries determined).  There are ${{n-2} \choose {d}}$ such choices.  Adding these three counts gives us the desired recursion from Equation (\ref{eq:recursion}).

We next check the initial conditions.  For $n=5$ Hochster's formula again gives us $\beta^5_{1,2} = \beta^5_{2,3} = 5$.  One can check (see Example \ref{ex:tableaux}) that there are precisely 5 standard Young tableau of shape $(2,2,1)$ and of its conjugate shape $(3,2)$.

For arbitrary $n$ and $d = 1$ we have $\beta^n_{1,2} = {n \choose 2} - n$, given by the number of generators of $J_n$.  On the other hand, in a standard Young tableau of shape $(2,2,1^{n-4})$ we can have any pair $(i,j)$ with $i<j$ occupy the second row except $(1,2), (1,3), \dots, (1,n)$, or $(2,3)$.  Hence the number of such fillings is also given by ${n \choose 2} - n$.

Similarly, for arbitrary $n$ and $d=n-3$, Hochster's formula implies that the Betti numbers $\beta^n_{n-3, n-2}$ are given by all choices of $n-2$ vertices of the $n$-gon corresponding to complements of diagonals (since these remaining pair of vertices with be disconnected).  Hence again $\beta^n_{n-3,n-2} = {n \choose 2} - n$ (this also follows from the fact that the ring $R/J_n$ is Gorenstein and therefore has a palindromic sequence of Betti numbers).  In terms of tableaux, we see that the shape $(n-2,2)$ is conjugate to $(2,2,1^{n-4})$ and hence both shapes have the same number of fillings.   
\end{proof}

\begin{remark}\label{rem:Bettiexplicit}
An application of the hook length formula gives an explicit value for the Betti numbers of $R/J_n$:
\begin{equation}
\beta^n_{d,d+1} = {n \choose {d+1}} \frac{d(n-d-2)}{n-1}
\end{equation}
\end{remark}

After a version this paper was posted on the ArXiv it was pointed out to the author that this formula had previously been established in \cite{BruHib}, with a combinatorial proof given in \cite{ChoKim}.

\begin{remark}
As we have seen, the rings $R/J_n$ are Gorenstein and hence the Betti numbers of $R/J_n$ are palindromic in the sense that 
\[ \beta^n_{d} = \beta^n_{n-d-2}. \]
\noindent
The realization of the Betti numbers of $R/J_n$  in terms of standard Young tableaux (Theorem \ref{thm:Betti}) provides a nice combinatorial interpretation of this property.  The partition $(d+1,2,1^{n-d-3})$ is conjugate to the partition $((n-d-2)+1, 2, 1^{n-(n-d-2)-3}) = (n-d-1, 2, 1^{d-1})$ and hence they have the same number of fillings. 

\end{remark}

\begin{example}
For $n =6$ the resolution of $R/J_6$ can be represented as
\[ \begin{array}{ccccccccccccc}
0 & \leftarrow & R & \leftarrow & \yng(2,2,1,1) & \leftarrow & \yng(3,2,1) & \leftarrow & \yng(4,2) & \leftarrow & R & \leftarrow & 0. \end{array}
\]
%
\noindent
In each homological degree we have a basis for the free module given by all standard Young tableaux of the indicated shape. Note that $(2,2,1,1)$ is conjugate to $(4,2)$.
\end{example}










\section{Discrete Morse theory and matchings}\label{sec:matching}

As we have seen, the associahedron $A_n$ (with the monomial labeling described above) supports a resolution of the ideal $J_n$.  We have also seen that the resolution is not minimal, and in particular the labeling of $A_n$ produces distinct faces $F \subsetneq G$ with the same monomial labeling.  In fact as $n$ increases the resolution becomes further and further from minimal in the sense that the number of facets of $A_n$ (a Catalan number, on the order of $\frac{4^n}{n^{3/2}}$) dominates the dimension of the second highest syzygy module of $R/J_n$ (which is on the order of $n^2$).  

\begin{example}
Face numbers versus Betti numbers for $n = 6,7,8,9$ are indicated below.  Here $f(n,j)$ refers to the number of $j$-dimensional faces in the Associahedron $A_n$.
\medskip

\hspace{1 in}
\begin{tabular}{|c|ccccc|}
  \hline
   & d=0 & 1 & 2 & 3 & 4 \\
  \hline
  $\beta^6_d$ & 1 & 9 & 16 & 9 & 1 \\
  $f(6, d-1)$ & 1 & 9 & 21 & 14 & 1 \\
  \hline
\end{tabular}

\medskip

\hspace{1 in}
\begin{tabular}{|c|cccccc|}
  \hline
   & d=0 & 1 & 2 & 3 & 4 & 5 \\
  \hline
  $\beta^7_d$ & 1 & 14 & 35 & 35 & 14 & 1 \\
  $f(7, d-1)$ & 1 & 14 & 56 & 84 & 42 & 1 \\
  \hline
\end{tabular}

\medskip
\hspace{1 in}
\begin{tabular}{|c|ccccccc|}
  \hline
   & d=0 & 1 & 2 & 3 & 4 & 5 & 6 \\
  \hline
  $\beta^8_d$ & 1 & 20 & 64 & 90 & 64 & 20 & 1 \\
  $f(8, d-1)$ & 1 & 20 & 120 & 300 & 330 & 132 &1 \\
  \hline
\end{tabular}

\medskip
\hspace{1 in}
\begin{tabular}{|c|cccccccc|}
  \hline
   & d=0 & 1 & 2 & 3 & 4 & 5 & 6 & 7 \\
  \hline
  $\beta^9_d$ & 1 & 27 & 105 & 189 & 189 & 105 & 27 & 1 \\
  $f(9, d-1)$ & 1 & 27 & 225 & 825 & 1485 & 1287 & 429 & 1  \\
  \hline
\end{tabular}

\end{example}

\subsection{Morse matchings and first syzygies}

Batzies and Welker and others (see \cite{BatWel} and \cite{Sko}) have developed a theory of algebraic Morse theory that allows one to match faces of a labeled complex in order to produce resolutions that become closer to minimal.  In the usual combinatorial description of this theory, one must match elements in the face poset of the labeled complex that have the same monomial labeling.  The matching must also satisfy a certain acyclic condition, described below. We refer to \cite{BatWel} for further details.

A closer analysis of our monomial labeling of $A_n$ reveals certain faces that must be matched away in any minimal resolution, in the sense that the associated monomial has the wrong degree.  In particular, since we know that $R/J_n$ has an `almost' linear resolution (as described above) it must be the case that in any minimal cellular resolution ${\mathcal X}$, each $j$-dimensional face of ${\mathcal X}$ is labeled by a monomial of degree $j+2$ (for $j < n-3$).  Our labeling of $A_n$ has the property that the monomial $m$ associated to a face $F$ is given by the product of the variables involved in the choice of diagonals, and in particular a `properly' labeled $j$-dimensional face corresponds to a subdivision of $C_n$ with $j+1$ diagonals involving precisely $j+2$ vertices.  This motivates the following.

\begin{definition}
Suppose $S$ is a subdivision of the $n$-gon $C_n$, by which we mean a collection of $d$ non-crossing diagonals.  We will say that $S$ is \emph{proper} if the set of endpoints of the diagonals has exactly $d+1$ elements (as vertices of $C_n$).  We will say that $S$ is \emph{superproper} if uses more than $d+1$ vertices and \emph{subproper} if it uses less.
\end{definition}

\begin{figure}[!ht]
\begin{center}
  \includegraphics[scale = .4]{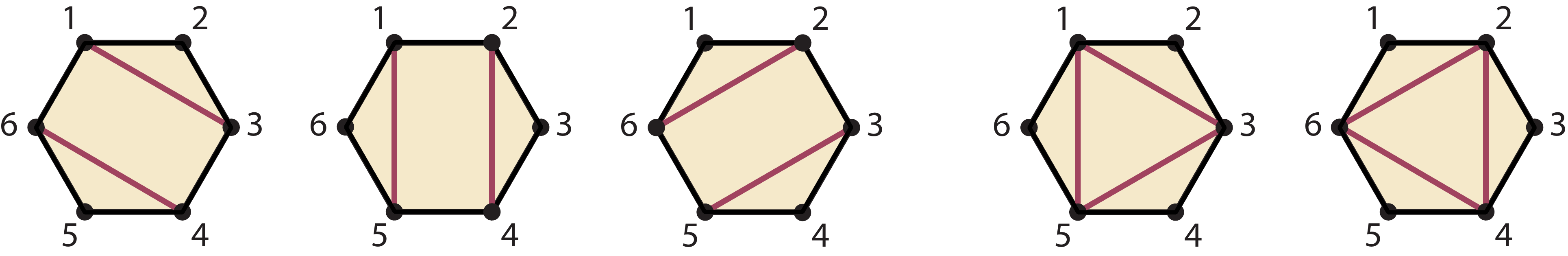}
\end{center}
    \caption{For $n=6$, the three superproper subdivisions with $d=2$ and the two subproper subdivisions with $d=3$.  All other subdivisions of the $6$-gon are proper.}
\label{Fig:improper}
\end{figure}

In fact we can explicitly describe a (partial) Morse matching on the monomial-labeled face poset of $A_n$ that is `perfect' for rank $d=2$.   A superproper subdivision of an $n$-cycle with $d=2$ is simply a pair of disjoint diagonals, say $E = \{ij, k \ell \}$ with $i<j$, $k<\ell$, and $i < k$.  In the face poset of $A_n$ we match this 2-face with the 3-face $F$, where
\[F = \left\{ \begin{array}{ll}
\{ij, k\ell, j \ell\} &  \textrm{if $j < k$} \\
\{ij, k \ell, i \ell\} & \textrm{otherwise.}
\end{array} \right.
\]

A subproper subdivision with $d=3$ is an inscribed triangle, say with diagonals $\{ij, ik, jk\}$, $i<j<k$.  We match this face with the (proper) 1-face $\{ij, jk\}$.  Recall that the Hasse diagram of the face poset of $A_n$ is a graph with vertices given by all faces of $A_n$, and with edges given by all cover relations $X \prec Y$.  It is easy to that our association is a matching of the Hasse diagram of the face poset of $A_n$, and it is clearly `algebraic' in the sense that matched faces have the same monomial labeling. 

As is typical we think of the Hasse diagram as a directed graph with the orientation on a matched edge pointing up (increasing dimension), and with all unmatched edges pointing down.  The collection of faces not involved in the matching are called the `critical cells', they form a subposet of the original poset.  The main theorem of (algebraic) discrete Morse theory \cite{BatWel} says that if we have an acyclic (algebraic) matching on the Hasse diagram of a cellular resolution, then the critical cells form a monomial-labeled CW-complex that also supports a cellular resolution.  In this way one can obtains a resolution that is closer to being minimal.  In our case we have the following result.

\begin{proposition}\label{prop:edges}
For all $n \geq 6$, the matching on the monomial labeled face poset of $A_n$ described above is acyclic.  Furthermore, the number of unmatched (critical) edges is given by $\beta^n_2$.
\end{proposition}

\begin{proof}
We first make the simple observation that if $F$ is any 2-face of $A_n$ corresponding to a subproper subdivision (in other words an inscribed triangle), there for any 1-face $E$ with $E \prec F$ we must have that $E$ is a path of length 2 (a proper subdivision with $d=2$).  Similarly, if $E$ is a path of length 2 and $E \prec F$ is an \emph{upward} oriented edge then it must be the case that $F$ is an inscribed triangle with the same vertex set as $E$.  This implies that there cannot be any cycles in the oriented face poset involving proper subdivisions with $d=2$ (paths of length 2).  

Next suppose $E \prec F$ is an upward oriented edge in the face poset of $A_n$ where $E$ consists of two disjoint diagonals (a superproper subdivision with $d=2$).  Then according to our matching it must be the case that $F$ is a path of length 3.  To form a cycle in the face poset there must be some downward edge from $F$ to $E^\prime$ with $E^\prime \prec F$.  But then according to our matching it must be the case that $E^\prime$ is a path of length 2.  Hence our observation from the previous paragraph implies that no cycles exist.  We conclude that the matching is acyclic.

We next count the unmatched edges. First observe that the number of proper subdivisions of an $n$-gon with 2 diagonals is given by $\frac{n(n-3)(n-4)}{2}$.  To see this, note that the diagonals involved in such a subdivision must form a path of length 3.  Once we designate the middle vertex in this path (of which there are $n$ choices) we have ${n-3 \choose 2}$ choices for the remaining two vertices.  Next we claim that the the number of subproper subdivisions of an $n$-gon with $d=3$ diagonals (necessarily forming an inscribed triangle) is given by $\frac{ n(n-4)(n-5)}{6}$.  To see this, we first count inscribed triangles with ordered vertex set $(v_1, v_2, v_3)$.  We are free to choose the first vertex $v_1$ from among the $n$ nodes of the cycle.  For the second vertex we have two cases.  If we choose $v_2$ from among the two vertices that are distance 2 from $v_1$, we are left with $(n-5)$ choices for $v_3$.  If we choose $v_2$ from among the vertices more than distance 2 from $v_1$ (of which there are $n-5$ choices), we are then left with $n-6$ choices for $v_3$.  In total there are
\[n(2(n-5) + (n-5)(n-6)) = n(n-4)(n-5)\]
inscribed triangles with the ordered vertex set.  Dividing out by $6$ to forget the ordering gives us the desired count.  As described above, we match each of the $d=2$ superproper subdivisions with a $d=3$ proper subdivision, and we match each of the $d=3$ subproper subdivisions with a $d=2$ proper subdivision.  Hence after matching the number of critical edges is given by
\[\frac{n(n-3)(n-4)}{2} - \frac{n(n-4)(n-5)}{6} = {n \choose 3} \frac{2(n-4)}{n-1},\]
which is precisely $\beta^n_2$ (see Remark \ref{rem:Bettiexplicit}).  This completes the proof.
\end{proof}

Hence our simple matching leaves precisely the number of critical 1-cells that we require.  The rank of the first free module in the resulting cellular resolution will be equal to the rank of the first syzygy module of $R/J_n$.  

\begin{example}
For $n=6$ this matching in fact leads to a minimal resolution of $J_6$.  In this case we have three superproper subdivisions with $d=2$ (namely $\{13, 46\}$, $\{15,24\}$,  and $\{26, 35\}$), and two subproper subdivisions with $d=3$ (namely $\{13, 15, 35\}$ and $\{24, 26, 46\}$).  


\end{example}

\begin{figure}[!ht]
\begin{center}
  \includegraphics[scale = .32]{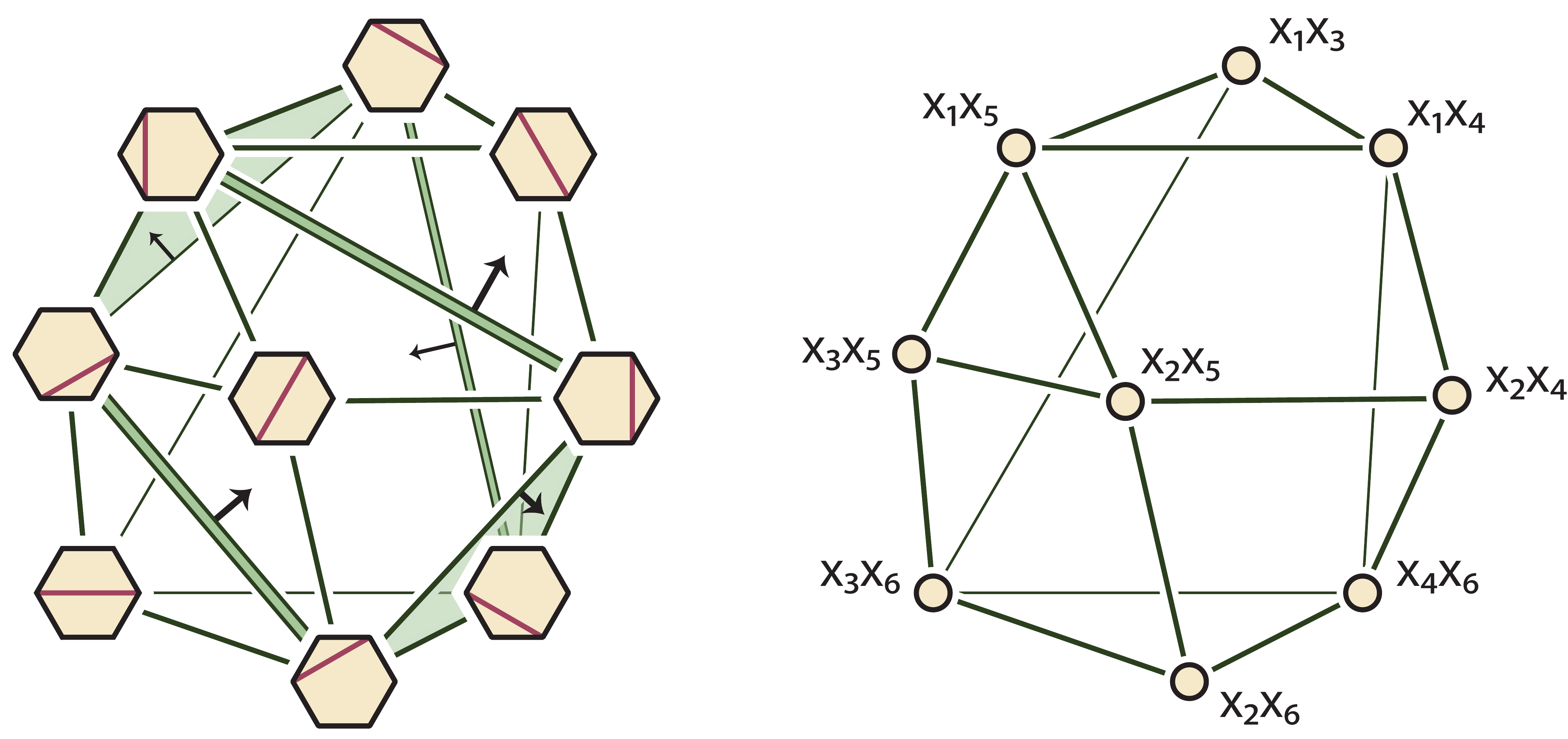}
\end{center}
    \caption{The monomial labeled $A_6$ with five pairs of faces matched (the shaded faces are the improper subdivisions).  The resulting complex on the right supports a minimal resolution of $J_6$.}
\label{Fig:matching}
\end{figure}

We remark that the procedure described above can be extended to the case $n=7$.  We leave the details to the reader but point out that in this case we have:
\begin{itemize}
\item
$14$ superproper subdivisions with $d=2$ (pairs of disjoint diagonals), corresponding to 14 edges of $A_7$ that we match with $2$-faces,
\item
$7$ subproper subdivisions with $d = 3$ (inscribed triangles in a $7$-gon), corresponding to seven $2$-faces that each get matched with an edge,
\item
$14$ superproper subdivisions with $d=3$ (forests consisting of three edges and two components), corresponding to seven $2$-faces that get matched to a $3$-face,
\item
$14$ subproper subdivisions with $d=4$ (inscribed triangles with a pendant edge), corresponding to fourteen $3$-faces that get matched down to a $2$-face.
\end{itemize}

The resulting CW-complex has $56 - 14 - 7 = 35$ edges, $84 - 14 - 7 - 14 - 14 = 35$ two-dimensional faces, and $42 - 28 = 14$ three-dimensional faces, as desired.  Unfortunately we do not know how to extend this matching procedure in general; see the next Section for some comments regarding this.

\subsection{An involution of the associahedron tableaux}

Recall that the faces of the associahedron $A_n$ are counted by standard Young tableaux of certain shapes, while the Betti numbers of $J_n$ are counted by standard Young tableaux  of certain subshapes.  Again motivated by discrete Morse theory this leads to ask whether we can find a matching on the set of associahedron tableaux such that the unmatched elements correspond to the Betti numbers of $J_n$.  This matching should have the property that two matched tableaux differ in cardinality by one.  Let us emphasize that since we do not have a poset structure on these elements we are not at this pointing searching for a {\emph Morse} matching.  Let us first fix some notation.

\begin{definition}
For fixed $n \geq 4$ and $1 \leq d \leq n-3$, we call the collection of standard Young tableaux of shape $(d+1,d+1, 1^{n-d-3})$ the \emph{associahedron tableaux} (denoted by ${\mathcal A_n}$), and the standard Young tableaux of shape $(d+1, 2, 1^{n-d-3})$ the \emph{syzygy tableaux} (denoted by ${\mathcal S_n}$).   Let ${\mathcal A} = \bigcup {\mathcal A}_n$ and ${\mathcal S} = \bigcup{\mathcal S}_n$.

\end{definition}

Note that an element of ${\mathcal A}_n$ has $n+d-1$ boxes, whereas an element of ${\mathcal S}_n$ has $n$ boxes.  If $X \in {\mathcal A}_n$ is an associahedron tableux with largest entries in the second row in the positions $(*,*, n+1, n+2, \dots, n+d-1)$ then it naturally becomes a syzygy tableau by just removing those boxes.  In particular we say that these particular associahedron tableau \emph{restrict} to syzygy tableaux, and in this way we have a natural inclusion ${\mathcal S}_n \subset {\mathcal A}_n$.  

\begin{example}
The associahedron tableau on the left restricts to a syzygy tableau, whereas the associahedron tableau on the right does not. Here $n=7$ and $d=3$. 
\[\young(****,**89,*) \; \rightarrow \; \young(****,**,*) \hspace{.75 in} \young(****,***8,9) \]
\end{example}

We next describe an involution on the set ${\mathcal A}$ such that the fixed elements are precisely the elements that restrict to ${\mathcal S}$.  If $X$ is a standard young tableau we use $|X|$ to denote the number of boxes in the underlying partition.

\begin{proposition}\label{prop:matching}
There exists an involution $\sigma$ on the set ${\mathcal A}$ such that the fixed point set of $\sigma$ is precisely the set of tableaux in ${\mathcal A}$ that restrict to ${\mathcal S}$. Furthermore, if $X \in {\mathcal A}$ such that $\sigma(X) \neq X$, we have $|\sigma(X)| = |X| \pm 1$.
\end{proposition}

\begin{proof}
Suppose $X \in {\mathcal A}_n$ is an associahedron tableau.  If $X$ restricts to a syzygy tableau we set $\sigma(X) = X$. Otherwise some element of $\{n+1, n+2, \dots, n+d-1\}$ is not in the second row of $X$; let $i$ be the largest element with this property.  Then $i$ must be the last element of the first row, or else the bottom element in the first column.

In the latter case ($i$ is the bottom most element of first column) we bring that element $i$ to the first row, and add the element $n+d$ to the end of the second row.  This defines $\sigma(X)$.   In the former case ($i$ is the last element of the first row) we obtain $\sigma(X)$ by bringing that element down to the bottom of the first column and deleting the last element of the second row (which must be $n+d-1$).  It is clear that $\sigma(\sigma(X)) = X$.
\end{proof}

\begin{example}\label{ex:matching}
An example of the involution matching an associahedron tableau of shape $(3,3,1,1)$ with one of shape $(4,4,1)$ is given by the following.

\[ \young(***,***,*,8) \; \leftrightarrow \; \young(***8,***9,*)\]
\end{example}

\section{Further questions}

We end with a number of questions that arise from our study.  As we have seen in Section \ref{sec:tableaux}, the number $f(n,d)$ of dissections of an $n$-gon using $d$ (non-crossing) diagonals is well understood, and is given by the number of standard Young tableaux of shape $(d+1,d+1,1^{n-d-3})$.   In the context of enumerating the Betti numbers of the ideal $J_n$ we were interested in subdivisions that involved a fixed number of vertices.  Define $f(n,d,j)$ to be the number of ways to choose $d$ non-crossing diagonals in a convex $n$-gon such that the set of endpoints consists of precisely $j$ vertices of the $n$-gon. 

\begin{question}
 Is there a nice formula for $f(n,d,j)$?  Can it be related to the standard Young tableaux of shape $(d+1, d+1, 1^{n-d-3})$?  
\end{question}

\noindent
We note that if we take $d = n-3$ then varying $j$ gives a refinement of the Catalan numbers which (as far as we know) has not appeared elsewhere. The first few refinements are 
\[14 = 2 + 12, \hspace{.1 in} 42 = 14 + 28, \hspace{.1 in} 132 = 4 + 64 + 64, \dots\]

A related question would be to consider those subdivisions for which the collection of diagonals forms a (connected) tree, since this is likely the more relevant property in the context of syzygies.  For $n \leq 7$ it so happens that the proper subdivisions correspond to those collections of diagonals that form a tree. However, for $n=8$ there exist proper subdivisions that are not trees: for example if $d = 4$ we can take 3 diagonals to form a triangle with vertices $\{1,3,5\}$ along with one disconnected diagonal $\{6,8\}$, in total using 5 vertices of the $8$-gon.

\begin{question}
How many dissections of an $n$-gon with $d$ diagonals have the property that the set of diagonals forms a tree?
\end{question}

In our quest for a Morse matching on the monomial labeled face poset of the associahedron $A_n$ we were unable to employ Stanley's bijection between faces of $A_n$ and standard Young tableaux.   As mentioned above, the difficulty arises as the bijection given in \cite{Sta} is recursively defined and involves certain choices.  However, the fact that the face poset of $A_n$ is labeled by standard Young tableaux suggests that there might be a meaningful poset structure on the set of all standard Young tableaux (or at least the set of Associahedron tableaux).  The hope would be that this poset structure extends the partial order given by the involution on ${\mathcal A}$ described in the proof of Proposition \ref{prop:matching}.  Hence the poset should be graded by the number of boxes in the underlying partition, but will \emph{not} restrict to Young's lattice if one forgets the fillings.  We refer to Example \ref{ex:matching} for a example of a cover relation between two standard Young tableaux such that the underlying partitions are not related in Young's lattice.

\begin{question}
Does there exist a meaningful poset structure on the set of standard Young tableaux, consistent with the conditions described above? 
\end{question}

Finally, we see in Figure \ref{Fig:matching} that a minimal resolution of $J_6$ is supported on a 3-dimensional polytope.  As we mentioned the construction there was a bit ad hoc but it does lead us to following:
\begin{question}
Does the ideal $J_n$ have a minimal cellular resolution supported on a (necessarily $(n-3)$-dimensional) polytope?
\end{question}

Work in this direction (along with some further generalizations) is currently being pursued by Engstr\"om and Linusson \cite{EngLin}.

\subsection*{Acknowledgements}

We thank Ken Baker for his assistance with the figures, and Alex Engstr\"om, Jakob Jonsson, and Michelle Wachs for helpful conversations. Alex and I first realized the potential connection to standard Young tableaux after inputting the Betti numbers of $J_n$ into OEIS \cite{OEIS} some years ago.  Thanks also to the anonymous referee for a careful reading.

\end{document}